\documentclass[reqno,11pt]{article} %leqno is the option to put formula numbers on the left side
\usepackage[margin=1in]{geometry}  % set the margins to 1in on all sides
\usepackage{amsmath, amssymb}
\usepackage{amsthm} %theorem environment option
\newcommand{\Mod}[1]{\ (\textup{mod}\ #1)}
\theoremstyle{plain} %text of this environment is typesetted in italics
\newtheorem{theorem}{\indent\sc Theorem}[section]
\newtheorem{lemma}[theorem]{\indent\sc Lemma}
\newtheorem{corollary}[theorem]{\indent\sc Corollary}
\newtheorem{proposition}[theorem]{\indent\sc Proposition}

\theoremstyle{definition} %text of this environment is typesetted in roman letters
\newtheorem{definition}[theorem]{\indent\sc Definition}
\newtheorem{assumption}[theorem]{\indent\sc Assumption}

\newtheorem{remark}[theorem]{\indent\sc Remark}

\makeatletter
\def\address#1#2{\begingroup
\noindent\parbox[t]{7.8cm}{%
\small{\scshape\ignorespaces#1}\par\vskip1ex
\noindent\small{\itshape E-mail address}%
\/: #2\par\vskip4ex}\hfill%
\endgroup}%
\makeatother
\title{Siegel families with application to class fields} %title of the paper
\author{
\textsc{Ja Kyung Koo, Dong Hwa Shin and Dong Sung Yoon$^*$} %names of authors
}
\date{} %leave empty
\begin{document}

\allowdisplaybreaks

\maketitle

%%%%%%%%%%%%%%% footnote %%%%%%%%%%%%%%%%
\footnote{ %2010 MSC numbers
}
\footnote{ %key words and phrases
}
\footnote{$^*$Corresponding author.}
\footnote{
\thanks{The second named author was
supported by Hankuk University of Foreign Studies Research Fund of
2016.} }
%%%%%%%%%%%%%%%%%%%%%%%%%%%%%%%%%%%%%%%

\begin{abstract}
We investigate certain families of meromorphic Siegel modular functions
on which Galois groups act in a natural way. By using Shimura's reciprocity law we construct some algebraic numbers in the ray class fields of CM-fields in terms of special values of functions in these Siegel families.
\end{abstract}

\section {Introduction}

For a positive integer $N$ let $\mathfrak{F}_N$ be the field of meromorphic modular functions of level $N$ (defined on $\mathbb{H}=
\{\tau\in\mathbb{C}~|~\mathrm{Im}(\tau)>0\}$) whose Fourier coefficients belong to the $N$th cyclotomic field. As is well known, $\mathfrak{F}_N$ is a Galois extension of $\mathfrak{F}_1$ whose Galois group is isomorphic to $\mathrm{GL}_2(\mathbb{Z}/N\mathbb{Z})/\{\pm I_2\}$
(\cite[$\S$6.1--6.2]{Shimura71}). Now, let $N\geq2$ and consider a set
\begin{equation*}
V_N=\{\mathbf{v}\in\mathbb{Q}^2~|~\textrm{$N$ is the
smallest positive integer for which}~N\mathbf{v}\in\mathbb{Z}^2\}
\end{equation*}
as the index set. We call a family $\{f_\mathbf{v}(\tau)\}_{\mathbf{v}\in V_N}$
of functions in $\mathfrak{F}_N$ a \textit{Fricke family} of level $N$ if
each $f_\mathbf{v}(\tau)$ depends only on $\pm\mathbf{v}\Mod{\mathbb{Z}^2}$ and satisfies
\begin{equation*}
f_\mathbf{v}(\tau)^\alpha=
f_{\alpha^T\mathbf{v}}(\tau)\quad(\alpha\in\mathrm{GL}_2(\mathbb{Z}/N\mathbb{Z})/\{\pm I_2\}),
\end{equation*}
where $\alpha^T$ means the transpose of $\alpha$.
For example, Siegel functions of one-variable form such a Fricke family of level $N$
(\cite[Proposition 1.3 in Chapter 2]{K-L}).
See also \cite{J-K-S} or \cite{K-Y}.
\par
Let $K$ be an imaginary quadratic field
with the ring of integers $\mathcal{O}_K$, and let $\mathfrak{f}$
be a proper nontrivial ideal of $\mathcal{O}_K$. We denote by $\mathrm{Cl}(\mathfrak{f})$ and $K_\mathfrak{f}$
the ray class group modulo $\mathfrak{f}$ and its corresponding ray class field
modulo $\mathfrak{f}$, respectively.
If $\{f_\mathbf{v}(\tau)\}_\mathbf{v}$
is a Fricke family of level $N$ in which every $f_\mathbf{v}(\tau)$ is holomorphic on $\mathbb{H}$, then we can assign to each ray class $\mathcal{C}\in\mathrm{Cl}(\mathfrak{f})$
an algebraic number $f_\mathfrak{f}(\mathcal{C})$ as a special value of a function in  $\{f_\mathbf{v}(\tau)\}_\mathbf{v}$. Furthermore, we attain by Shimura's reciprocity law that $f_\mathfrak{f}(\mathcal{C})$ belongs to $K_\mathfrak{f}$ and satisfies
\begin{equation*}
f_\mathfrak{f}(\mathcal{C})^{\sigma_\mathfrak{f}(\mathcal{D})}=f_\mathfrak{f}(\mathcal{CD})
\quad(\mathcal{D}\in\mathrm{Cl}(\mathfrak{f})),
\end{equation*}
where $\sigma_\mathfrak{f}$ is the Artin reciprocity map for $\mathfrak{f}$
(\cite[Theorem 1.1 in Chapter 11]{K-L}).
\par
In this paper, we shall define a Siegel family $\{h_M(Z)\}_M$ of level $N$
consisting of meromorphic Siegel modular functions of (higher) genus $g$ and level $N$,
which would be a generalization of a Fricke family of level $N$ in case $g=1$
(Definition \ref{Siegelfamily}).
It turns out that every Siegel family of level $N$ is induced from a
meromorphic Siegel modular function for the congruence subgroup $\Gamma^1(N)$
(Theorem \ref{phiNisomorphism}).
\par
Let $K$ be a CM-field and let $\mathfrak{f}=N\mathcal{O}_K$.
Given a Siegel family $\{h_M(Z)\}_M$ of level $N$,
we shall introduce a number $h_\mathfrak{f}(\mathcal{C})$
by a special value of a function in $\{h_M(Z)\}_M$
for each ray class $\mathcal{C}\in\mathrm{Cl}(\mathfrak{f})$ (Definition \ref{invariant}).
Under certain assumptions on $K$ (Assumption \ref{assumption})
we shall prove that if $h_\mathfrak{f}(\mathcal{C})$ is finite, then it
lies in the ray class field $K_\mathfrak{f}$ whose Galois conjugates are
of the same form (Theorem \ref{main} and Corollary \ref{Cor}).
To this end, we assign a principally polarized abelian variety
to each nontrivial ideal of $\mathcal{O}_K$, and apply Shimura's reciprocity law to
$h_\mathfrak{f}(\mathcal{C})$.

\section {Actions on Siegel modular functions}

First, we shall describe the Galois group between fields of meromorphic Siegel modular functions in a concrete way.
\par
Let $g$ be a positive integer, and let $\eta_g=\left[\begin{matrix}O_g&-I_g\\I_g&O_g\end{matrix}\right]$. For every commutative ring $R$ with unity
we denote by
\begin{eqnarray*}
\mathrm{GSp}_{2g}(R)&=&\left\{\alpha\in \mathrm{GL}_{2g}(R)~|~\alpha^T \eta_g\alpha=\nu(\alpha)\eta_g~\textrm{with}~\nu(\alpha)\in R^\times\right\},\\
\mathrm{Sp}_{2g}(R)&=&\{\alpha\in\mathrm{GSp}_{2g}(R)~|~\nu(\alpha)=1\}.
\end{eqnarray*}
Let
\begin{equation*}
G=\mathrm{GSp}_{2g}(\mathbb{Q}),
\end{equation*}
and let
$G_{\mathbb{A}}$ be the adelization of $G$, $G_0$ its non-archimedean part and $G_\infty$ its archimedean part.
One can extend the multiplier map $\nu:G\rightarrow\mathbb{Q}^\times$ continuously to the map $\nu:G_\mathbb{A}\rightarrow\mathbb{Q}_\mathbb{A}^\times$, and set
\begin{eqnarray*}
G_{\infty+}=\{\alpha\in G_\infty~|~\nu(\alpha)>0\},\quad
G_{\mathbb{A}+}=G_0G_{\infty+},\quad
G_+=G\cap G_{\mathbb{A}+}.
\end{eqnarray*}
Furthermore, let
\begin{eqnarray*}
\Delta&=&\left\{\left[\begin{matrix}I_g&O_g\\O_g&s I_g\end{matrix}\right]
~|~s\in\prod_p\mathbb{Z}_p^\times\right\},\\
U_1&=&\prod_p\mathrm{GSp}_{2g}(\mathbb{Z}_p)\times G_{\infty+},\\
U_N&=&\{x\in U_1~|~x_p\equiv I_{2g}\Mod{N\cdot M_{2g}(\mathbb{Z}_p)}~\textrm{for all rational primes $p$}\}
\end{eqnarray*}
for every positive integer $N$.
Then we have
\begin{equation*}
U_N\unlhd U_1\leq G_{\mathbb{A}+}\quad\textrm{and}\quad
G_{\mathbb{A}+}=U_N\Delta G_{+}
\end{equation*}
(\cite[Lemma 8.3 (1)]{Shimura00}).
\par
Note that the group $G_{\infty+}$
acts on the Siegel upper half-space $\mathbb{H}_g=\{Z\in M_g(\mathbb{C})~|~
Z^T=Z,~\mathrm{Im}(Z)~\textrm{is positive definite}\}$ by
\begin{equation*}
\alpha(Z)=
(AZ+B)(CZ+D)^{-1}\quad(\alpha\in G_{\infty+},~Z\in\mathbb{H}_g),
\end{equation*}
where $A,B,C,D$ are $g\times g$ block matrices of $\alpha$.
Let $\mathcal{F}_N$ be the field
of meromorphic Siegel modular functions of genus $g$ for the congruence subgroup
\begin{equation*}
\Gamma(N)=\left\{\gamma\in\mathrm{Sp}_{2g}(\mathbb{Z})~|~
\gamma\equiv I_{2g}\Mod{N\cdot M_{2g}(\mathbb{Z})}\right\}
\end{equation*}
of the symplectic group $\mathrm{Sp}_{2g}(\mathbb{Z})$ whose Fourier coefficients belong to the $N$th cyclotomic field $\mathbb{Q}(\zeta_N)$ with $\zeta_N=e^{2\pi i/N}$.
That is, if $f\in\mathcal{F}_N$, then
\begin{equation*}
f(Z)=\frac{\sum_{h}c(h)e(\mathrm{tr}(hZ)/N)}{\sum_{h}d(h)e(\mathrm{tr}(hZ)/N)}\quad\textrm{for some}~c(h),d(h)\in \mathbb{Q}(\zeta_N),
\end{equation*}
where the denominator and numerator of $f$ are Siegel modular forms of the same weight, $h$ runs over all $g\times g$ positive semi-definite symmetric matrices over half integers with integral diagonal entries,
and $e(w)=e^{2\pi i w}$ for $w\in\mathbb{C}$ (\cite[Theorem 1 in $\S$4]{Klingen}).
Let
\begin{equation*}
\mathcal{F}=\bigcup_{N=1}^\infty \mathcal{F}_N.
\end{equation*}

\begin{proposition}\label{Siegel-action}
There exists a homomorphism $\tau:G_{\mathbb{A}+}\rightarrow \mathrm{Aut}(\mathcal{F})$ satisfying the following properties: Let
$f(Z)=\frac{\sum_{h}c(h)e(\mathrm{tr}(hZ)/N)}{\sum_{h}d(h)e(\mathrm{tr}(hZ)/N)}\in\mathcal{F}_N$.
\begin{itemize}
\item[\textup{(i)}] If $\alpha\in G_{+}=
\{\alpha\in G~|~\nu(\alpha)>0\}$, then
\begin{equation*}
f^{\tau(\alpha)}=f\circ\alpha.
\end{equation*}
\item[\textup{(ii)}]
If $\beta=\left[\begin{matrix}I_g&O_g\\O_g&s I_g\end{matrix}\right]
\in\Delta$ and $t$ is a positive integer such that $t\equiv s_p \pmod{N\mathbb{Z}_p}$
for all rational primes $p$,
then
\begin{equation*}
f^{\tau(\beta)}=\frac{\sum_{h}c(h)^\sigma e(\mathrm{tr}(hZ)/N)}{\sum_{h}d(h)^\sigma e(\mathrm{tr}(hZ)/N)},
\end{equation*}
where $\sigma$ is the automorphism of $\mathbb{Q}(\zeta_N)$ given by $\zeta_N^\sigma=\zeta_N^t$.
\item[\textup{(iii)}] For every positive integer $N$ we have
\begin{equation*}
\mathcal{F}_N=\{f\in\mathcal{F}~|~ f^{\tau(x)}=f~\textrm{for all}~x\in U_N\}.
\end{equation*}
\item[\textup{(iv)}] $\ker(\tau)=\mathbb{Q}^\times G_{\infty+}$.
\end{itemize}
\end{proposition}
\begin{proof}
See \cite[Theorem 8.10]{Shimura00}.
\end{proof}

Since
\begin{equation*}
U_N(\mathbb{Q}^\times G_{\infty+})/\mathbb{Q}^\times G_{\infty+}\simeq U_N/(U_N\cap\mathbb{Q}^\times G_{\infty+})\simeq
\left\{\begin{array}{ll}
U_1/\pm G_{\infty+} & \textrm{if $N=1$},\\
U_N/G_{\infty+} & \textrm{if $N>1$},
\end{array}\right.
\end{equation*}
we see by Proposition \ref{Siegel-action} (iii) and (iv) that $\mathcal{F}_N$ is a Galois extension of $\mathcal{F}_1$ with
\begin{equation}\label{Galoisgroup}
\mathrm{Gal}(\mathcal{F}_N/\mathcal{F}_1)\simeq U_1/\pm U_N.
\end{equation}

\begin{proposition}\label{functionfields}
We have
\begin{equation*}
\mathrm{Gal}(\mathcal{F}_N/\mathcal{F}_1)\simeq \mathrm{GSp}_{2g}(\mathbb{Z}/N\mathbb{Z})/\{\pm I_{2g}\}.
\end{equation*}
\end{proposition}
\begin{proof}
Let $\alpha\in U_1$. Take a matrix $A$ in $M_{2g}(\mathbb{Z})$ for which $A\equiv\alpha_p\Mod{N\cdot M_{2g}(\mathbb{Z}_p)}$ for all rational primes $p$.
Define a matrix $\psi(\alpha)\in M_{2g}(\mathbb{Z}/N\mathbb{Z})$ by the image of $A$ under the natural reduction $M_{2g}(\mathbb{Z})\rightarrow M_{2g}(\mathbb{Z}/N\mathbb{Z})$.
Then by the Chinese remainder theorem $\psi(\alpha)$ is well defined and independent of the choice of $A$.
Furthermore, let $t$ be an integer relatively prime to $N$ such that $t\equiv \nu(\alpha_p)\Mod{N\mathbb{Z}_p}$ for all rational primes $p$.
We then derive that
\begin{equation*}
t\eta_g\equiv\nu(\alpha_p)\eta_g\equiv \alpha_p^T \eta_g \alpha_p\equiv A^T\eta_g A\equiv\psi(\alpha)^T\eta_g\psi(\alpha)\Mod{N\cdot M_{2g}(\mathbb{Z}_p)}
\end{equation*}
for all rational primes $p$, and hence $\psi(\alpha)\in\mathrm{GSp}_{2g}(\mathbb{Z}/N\mathbb{Z})$.
Thus we obtain a group homomorphism
\begin{equation*}
\psi:U_1\rightarrow\mathrm{GSp}_{2g}(\mathbb{Z}/N\mathbb{Z}).
\end{equation*}
\par
Let $\beta\in\mathrm{GSp}_{2g}(\mathbb{Z}/N\mathbb{Z})$, and take a preimage $B$ of $\beta$ under the natural reduction $M_{2g}(\mathbb{Z})\rightarrow M_{2g}(\mathbb{Z}/N\mathbb{Z})$.
Since $\nu(\beta)\in(\mathbb{Z}/N\mathbb{Z})^\times$ and
\begin{equation*}
B^T \eta_g B\equiv \beta^T\eta_g \beta \equiv\nu(\beta)\eta_g\Mod{N\cdot M_{2g}(\mathbb{Z})},
\end{equation*}
$B$ belongs to $\mathrm{GSp}_{2g}(\mathbb{Z}_p)$ for every rational prime $p$ dividing $N$.
Let $\alpha=(\alpha_p)_p$ be the element of $\prod_p\mathrm{GSp}_{2g}(\mathbb{Z}_p)$ given by
\begin{equation*}
\alpha_p=\left\{
\begin{array}{ll}
B & \textrm{if $p\,|\,N$},\\
I_{2g}&\textrm{otherwise}.
\end{array}
\right.
\end{equation*}
We then see that $\alpha\in U_1$ and $\psi(\alpha)=\beta$.
Thus $\psi$ is surjective.
\par
Clearly, $U_N$ is contained in $\ker(\psi)$.
Let $\gamma\in \ker(\psi)$. Since
$\gamma_p\equiv I_{2g} \Mod{N\cdot M_{2g}(\mathbb{Z}_p)}$ for all rational primes $p$,
we get $\gamma\in U_N$, and hence $\ker(\psi)=U_N$.
Therefore $\psi$ induces an isomorphism $U_1/U_N\simeq \mathrm{GSp}_{2g}(\mathbb{Z}/N\mathbb{Z})$, from which we achieve by (\ref{Galoisgroup}) 
\begin{equation*}
\mathrm{Gal}(\mathcal{F}_N/\mathcal{F}_1)\simeq  U_1/\pm U_N \simeq\mathrm{GSp}_{2g}(\mathbb{Z}/N\mathbb{Z})/\{\pm I_{2g}\}.
\end{equation*}
\end{proof}

\begin{remark}\label{decomp}
We have the decomposition
\begin{equation*}
\mathrm{Gal}(\mathcal{F}_N/\mathcal{F}_1)\simeq \mathrm{GSp}_{2g}(\mathbb{Z}/N\mathbb{Z})/\{\pm I_{2g}\}\simeq G_N\cdot \mathrm{Sp}_{2g}(\mathbb{Z}/N\mathbb{Z})/\{\pm I_{2g}\},
\end{equation*}
where
\begin{equation*}
G_N=\left\{\left[
\begin{matrix}
I_g&O_g\\
O_g&\nu I_g
\end{matrix}\right]~|~\nu\in(\mathbb{Z}/N\mathbb{Z})^\times
\right\}.
\end{equation*}
By Proposition \ref{Siegel-action} one can describe the action of $\mathrm{GSp}_{2g}(\mathbb{Z}/N\mathbb{Z})/\{\pm I_{2g}\}$ on $\mathcal{F}_N$
as follows: \\
Let $f(Z)=\frac{\sum_{h}c(h)e(\mathrm{tr}(hZ)/N)}{\sum_{h}d(h)e(\mathrm{tr}(hZ)/N)}\in\mathcal{F}_N$.
\begin{itemize}
\item[\textup{(i)}]
An element
$\beta=\left[
\begin{matrix}
I_g&O_g\\
O_g&\nu I_g
\end{matrix}\right]$ of $G_N$ acts on $f$ by
\begin{equation*}
f^{\beta}=\frac{\sum_{h}c(h)^{\sigma} e(\mathrm{tr}(hZ)/N)}{\sum_{h}d(h)^{\sigma} e(\mathrm{tr}(hZ)/N)},
\end{equation*}
where $\sigma$ is the automorphism of $\mathbb{Q}(\zeta_N)$ satisfying $\zeta_N^{\sigma}=\zeta_N^\nu$.
\item[\textup{(ii)}]
 An element $\gamma$ of $\mathrm{Sp}_{2g}(\mathbb{Z}/N\mathbb{Z})/\{\pm I_{2g}\}$ acts on $f$ by
\begin{equation*}
f^\gamma=f\circ\gamma',
\end{equation*}
where $\gamma'$ is any preimage of $\gamma$ under the natural reduction $\mathrm{Sp}_{2g}(\mathbb{Z})\rightarrow \mathrm{Sp}_{2g}(\mathbb{Z}/N\mathbb{Z})/\{\pm I_{2g}\}$.
\end{itemize}
\end{remark}

\section {Siegel families of level $N$}

By making use of the description of $\mathrm{Gal}(\mathcal{F}_N/\mathcal{F}_1)$ in $\S$2
we shall introduce a generalization of a Fricke family in higher dimensional cases.
\par
Let $N\geq2$. For $\alpha\in M_{2g}(\mathbb{Z})$
we denote by $\widetilde{\alpha}$ its reduction modulo $N$.
Define a set
\begin{equation*}
\mathcal{V}_N=\left\{(1/N)\left[\begin{matrix}A^T\\B^T\end{matrix}\right]~|~
\alpha=
\left[\begin{matrix}A&B\\C&D
\end{matrix}\right]\in M_{2g}(\mathbb{Z})~\textrm{such that}~\widetilde{\alpha}\in\mathrm{GSp}_{2g}
(\mathbb{Z}/N\mathbb{Z})
\right\}.
\end{equation*}
Let $M$ be an element of $\mathcal{V}_N$
stemmed from $\alpha\in M_{2g}(\mathbb{Z})$ such that $\widetilde{\alpha}\in\mathrm{GSp}_{2g}(\mathbb{Z}/N\mathbb{Z})$, and let $\beta$
be an element of $M_{2g}(\mathbb{Z})$ satisfying $\widetilde{\beta}
\in\mathrm{GSp}_{2g}(\mathbb{Z}/N\mathbb{Z})$. 
Then it is straightforward that
$\beta^T M$ is also an element of $\mathcal{V}_N$ given by the product $\alpha\beta$.

\begin{definition}\label{Siegelfamily}
We call a family $\{h_M(Z)\}_{M\in\mathcal{V}_N}$
a \textit{Siegel family} of level $N$ if it satisfies the following properties:
\begin{itemize}
\item[(S1)] Each $h_M(Z)$ belongs to $\mathcal{F}_N$.
\item[(S2)] $h_M(Z)$ depends only on
$\pm M\Mod{M_{2g\times g}(\mathbb{Z})}$.
\item[(S3)] $h_M(Z)^\sigma=h_{\sigma^TM}(Z)$
for all $\sigma\in\mathrm{GSp}_{2g}(\mathbb{Z}/N\mathbb{Z})/\{\pm I_{2g}\}
\simeq\mathrm{Gal}(\mathcal{F}_N/\mathcal{F}_1)$.
\end{itemize}
By $\mathcal{S}_N$ we mean the set of such Siegel families of level $N$.
\end{definition}

\begin{remark}
Let $\{h_M(Z)\}_M\in \mathcal{S}_N$.
\begin{itemize}
\item[(i)] The property (S3)
yields a right action of the group $\mathrm{GSp}_{2g}(\mathbb{Z}/N\mathbb{Z})/\{\pm I_{2g}\}$ on  $\{h_M(Z)\}_M$.
\item[(ii)] Let $M=(1/N)\left[\begin{matrix}A^T\\B^T\end{matrix}\right]\in\mathcal{V}_N$, and so there is a matrix
$\alpha=\left[\begin{matrix}A&B\\C&D\end{matrix}\right]\in M_{2g}(\mathbb{Z})$ such that
$\widetilde{\alpha}\in\mathrm{GSp}_{2g}(\mathbb{Z}/N\mathbb{Z})$.
Considering $\widetilde{\alpha}$ as an element of $\mathrm{GSp}_{2g}(\mathbb{Z}/
N\mathbb{Z})/\{\pm I_{2g}\}$ we obtain
\begin{equation*}
h_{(1/N)\left[\begin{smallmatrix}
I_g\\O_g\end{smallmatrix}\right]}(Z)^{\widetilde{\alpha}}=
h_{(1/N)\alpha^T\left[\begin{smallmatrix}
I_g\\O_g\end{smallmatrix}\right]}(Z)=
h_M(Z).
\end{equation*}
Thus the action of $\mathrm{GSp}_{2g}(\mathbb{Z}/N\mathbb{Z})/\{\pm I_{2g}\}$
on $\{h_M(Z)\}_M$ is transitive.
\end{itemize}
\end{remark}

Let
\begin{equation*}
\Gamma^1(N)=\left\{\gamma\in\mathrm{Sp}_{2g}(\mathbb{Z})~|~
\gamma\equiv\left[\begin{matrix}I_g&O_g\\\mathrm{*}&I_g\end{matrix}\right]
\Mod{N\cdot M_{2g}(\mathbb{Z})}\right\},
\end{equation*}
and let $\mathcal{F}^1_N(\mathbb{Q})$ be the field of meromorphic
Siegel modular functions for $\Gamma^1(N)$ with rational Fourier coefficients.

\begin{lemma}\label{particular}
If $\{h_M(Z)\}_M\in\mathcal{S}_N$, then
$h_{\left[\begin{smallmatrix}(1/N)I_g\\O_g\end{smallmatrix}\right]}(Z)\in\mathcal{F}^1_N(\mathbb{Q})$.
\end{lemma}
\begin{proof}
For any $\gamma=\left[\begin{matrix}A&B\\C&D\end{matrix}\right]\in\Gamma^1(N)$
we deduce by (S2) and (S3) that
\begin{eqnarray*}
h_{\left[\begin{smallmatrix}(1/N)I_g\\O_g\end{smallmatrix}\right]}(\gamma(Z))
=h_{\left[\begin{smallmatrix}(1/N)I_g\\O_g\end{smallmatrix}\right]}(Z)^{\widetilde{\gamma}}=
h_{\gamma^T\left[\begin{smallmatrix}(1/N)I_g\\O_g\end{smallmatrix}\right]}(Z)
=h_{(1/N)\left[\begin{smallmatrix}A^T\\B^T\end{smallmatrix}\right]}(Z)
=h_{\left[\begin{smallmatrix}(1/N)I_g\\O_g\end{smallmatrix}\right]}(Z)
\end{eqnarray*}
because $A\equiv I_g,~B\equiv O_g\Mod{N\cdot M_g(\mathbb{Z})}$.
Thus $h_{\left[\begin{smallmatrix}(1/N)I_g\\O_g\end{smallmatrix}\right]}(Z)$ is modular
for $\Gamma^1(N)$.
\par
For every $\nu\in(\mathbb{Z}/N\mathbb{Z})^\times$
we see by (S2) and (S3) that
\begin{equation*}
h_{\left[\begin{smallmatrix}(1/N)I_g\\O_g\end{smallmatrix}\right]}(Z)^{\left[\begin{smallmatrix}
I_g&O_g\\O_g&\nu I_g\end{smallmatrix}\right]}
=
h_{\left[\begin{smallmatrix}
I_g&O_g\\O_g&\nu I_g\end{smallmatrix}\right]\left[\begin{smallmatrix}(1/N)I_g\\O_g\end{smallmatrix}\right]}(Z)
=h_{\left[\begin{smallmatrix}(1/N)I_g\\O_g\end{smallmatrix}\right]}(Z),
\end{equation*}
which implies that $h_{\left[\begin{smallmatrix}(1/N)I_g\\O_g\end{smallmatrix}\right]}(Z)$ has
rational Fourier coefficients. This proves the lemma.
\end{proof}

One can consider $\mathcal{S}_N$ as a field under the binary operations
\begin{eqnarray*}
\{h_M(Z)\}_M+\{k_M(Z)\}_M&=&
\{(h_M+k_M)(Z)\}_M,\\
\{h_M(Z)\}_M\cdot\{k_M(Z)\}_M&=&
\{(h_Mk_M)(Z)\}_M.
\end{eqnarray*}
By Lemma \ref{particular} we get the ring homomorphism
\begin{eqnarray*}
\phi_N~:~\mathcal{S}_N&\rightarrow&\mathcal{F}^1_N(\mathbb{Q})\\
\{h_M(Z)\}_M&\mapsto&h_{\left[\begin{smallmatrix}(1/N)I_g\\O_g\end{smallmatrix}\right]}(Z).
\end{eqnarray*}

\begin{lemma}\label{Sp}
If $M\in\mathcal{V}_N$, then there is $\gamma=\left[
\begin{matrix}A&B\\C&D\end{matrix}\right]\in M_{2g}(\mathbb{Z})$ such that
$\widetilde{\gamma}\in\mathrm{Sp}_{2g}(\mathbb{Z}/N\mathbb{Z})$ and
$M=(1/N)\left[\begin{matrix}A^T\\B^T\end{matrix}\right]$.
\end{lemma}
\begin{proof}
Let $\alpha=\left[\begin{matrix}A&B\\U&V\end{matrix}\right]\in M_{2g}(\mathbb{Z})$
such that $\widetilde{\alpha}\in\mathrm{GSp}_{2g}(\mathbb{Z}/N\mathbb{Z})$ and
$M=(1/N)\left[\begin{matrix}A^T\\B^T\end{matrix}\right]$.
In $M_{2g}(\mathbb{Z}/N\mathbb{Z})$, decompose $\widetilde{\alpha}$
as
\begin{equation*}
\widetilde{\alpha}=
\left[\begin{matrix}I_g&O_g\\O_g&\nu I_g\end{matrix}\right]
\left[\begin{matrix}A&B\\\nu^{-1}U&\nu^{-1}V\end{matrix}\right]
\quad\textrm{with}~\nu=\nu(\widetilde{\alpha})\in(\mathbb{Z}/N\mathbb{Z})^\times
\end{equation*}
so that $\left[\begin{matrix}A&B\\\nu^{-1}U&\nu^{-1}V\end{matrix}\right]$ belongs
to $\mathrm{Sp}_{2g}(\mathbb{Z}/N\mathbb{Z})$. Since
the reduction $\mathrm{Sp}_{2g}(\mathbb{Z})\rightarrow
\mathrm{Sp}_{2g}(\mathbb{Z}/N\mathbb{Z})$ is surjective (\cite{Rapinchuk}),
we can take $\gamma\in M_{2g}(\mathbb{Z})$ satisfying
$\widetilde{\gamma}=\left[\begin{matrix}A&B\\\nu^{-1}U&\nu^{-1}V\end{matrix}\right]$.
\end{proof}

\begin{theorem}\label{phiNisomorphism}
$\mathcal{S}_N$ and $\mathcal{F}^1_N(\mathbb{Q})$ are isomorphic
via $\phi_N$.
\end{theorem}
\begin{proof}
Since $\mathcal{S}_N$ and $\mathcal{F}^1_N(\mathbb{Q})$
are fields, it suffices to show that $\phi_N$ is surjective.
\par
Let $h(Z)\in\mathcal{F}^1_N(\mathbb{Q})$. For each
$M\in\mathcal{V}_N$, take any $\gamma
=\left[\begin{matrix}A&B\\C&D\end{matrix}\right]\in M_{2g}(\mathbb{Z})$ such that
$\widetilde{\gamma}\in\mathrm{Sp}_{2g}(\mathbb{Z}/N\mathbb{Z})$ and
$M=(1/N)\left[\begin{matrix}A^T\\B^T\end{matrix}\right]$ by using Lemma \ref{Sp}. And, set
\begin{equation*}
h_M(Z)=h(Z)^{\widetilde{\gamma}}.
\end{equation*}
We claim that $h_M(Z)$ is independent of the choice of $\gamma$. Indeed, if
$\gamma'=\left[\begin{matrix}A&B\\C'&D'\end{matrix}\right]\in M_{2g}(\mathbb{Z})$ such that
$\widetilde{\gamma'}\in\mathrm{Sp}_{2g}(\mathbb{Z}/N\mathbb{Z})$, then we attain in $M_{2g}(\mathbb{Z}/N\mathbb{Z})$ that
\begin{equation*}
\widetilde{\gamma'}\widetilde{\gamma}^{-1}=\left[\begin{matrix}
A&B\\C'&D'\end{matrix}\right]
\left[\begin{matrix}
D^T&-B^T\\-C^T&A^T\end{matrix}\right]
=\left[\begin{matrix}
I_g&O_g\\\mathrm{*}&I_g\end{matrix}\right]
\end{equation*}
by the fact $\widetilde{\gamma},
\widetilde{\gamma'}
\in\mathrm{Sp}_{2g}(\mathbb{Z}/N\mathbb{Z})$.
Let $\delta$ be an element of $\mathrm{Sp}_{2g}(\mathbb{Z})$
such that $\widetilde{\delta}=\widetilde{\gamma'}\widetilde{\gamma}^{-1}$.
We then achieve 
\begin{equation*}
h(Z)^{\widetilde{\gamma'}}=(h(Z)^{\widetilde{\gamma'}\widetilde{\gamma}^{-1}})^{\widetilde{\gamma}}
=h(\delta(Z))^{\widetilde{\gamma}}=
h(Z)^{\widetilde{\gamma}}
\end{equation*}
because $h(Z)$ is modular for $\Gamma^1(N)$ and $\delta\in\Gamma^1(N)$.
\par
Now, for any $\sigma=\left[\begin{matrix}P&Q\\R&S\end{matrix}\right]
\in\mathrm{GSp}_{2g}(\mathbb{Z}/N\mathbb{Z})/\{\pm I_{2g}\}$ with
$\nu=\nu(\sigma)$
we derive that
\begin{eqnarray*}
h_M(Z)^\sigma&=&h(Z)^{\widetilde{\gamma}\sigma}\\
&=&h(Z)^{\left[\begin{smallmatrix}
A&B\\C&D
\end{smallmatrix}\right]\left[\begin{smallmatrix}
P&Q\\R&S
\end{smallmatrix}\right]}\\
&=&h(Z)^{\left[\begin{smallmatrix}
I_g&O_g\\O_g&\nu I_g
\end{smallmatrix}\right]
\left[\begin{smallmatrix}
AP+BR&AQ+BS\\\nu^{-1}(CP+DR)&\nu^{-1}(CQ+DS)
\end{smallmatrix}\right]}\\
&=&h(Z)^{\left[\begin{smallmatrix}
AP+BR&AQ+BS\\\nu^{-1}(CP+DR)&\nu^{-1}(CQ+DS)
\end{smallmatrix}\right]}\quad
\textrm{since $h(Z)$ has rational Fourier coefficients}\\
&=&h_{\left[\begin{smallmatrix}
(AP+BR)^T\\
(AQ+BS)^T
\end{smallmatrix}\right]}(Z)\\
&=&h_{\left[\begin{smallmatrix}P^T&R^T\\Q^T&S^T\end{smallmatrix}\right]
\left[\begin{smallmatrix}A^T\\B^T\end{smallmatrix}\right]}(Z)\\
&=&h_{\sigma^T M}(Z).
\end{eqnarray*}
This shows that the family $\{h_M(Z)\}_M$ belongs to $\mathcal{S}_N$. Furthermore, since
\begin{equation*}
\phi_N(\{h_M(Z)\}_M)=h_{\left[\begin{smallmatrix}
(1/N)I_g\\O_g
\end{smallmatrix}\right]}(Z)
=h(Z)^{\left[\begin{smallmatrix}
I_g&O_g\\O_g&I_g
\end{smallmatrix}\right]}=h(Z),
\end{equation*}
$\phi$ is surjective as desired.
\end{proof}

\begin{remark}
\begin{itemize}
\item[(i)] By Proposition \ref{functionfields} and Remark \ref{decomp} we obtain
\begin{equation*}
\mathrm{Gal}(\mathcal{F}_N/\mathcal{F}^1_N(\mathbb{Q}))
\simeq G_N\cdot\left\{\gamma\in\mathrm{Sp}_{2g}(\mathbb{Z}/N\mathbb{Z})/\{\pm I_{2g}\}~|~
\gamma=\pm\left[\begin{matrix}I_g& O_g\\\mathrm{*}&I_g\end{matrix}\right]
\right\}.
\end{equation*}
\item[(ii)]
Let $\mathcal{F}_{1,N}(\mathbb{Q})$ be the field of meromorphic Siegel modular functions
for
\begin{equation*}
\Gamma_1(N)=\left\{\gamma\in\mathrm{Sp}_{2g}(\mathbb{Z})~|~
\gamma\equiv\left[\begin{matrix}I_g&\mathrm{*}\\O_g&I_g\end{matrix}\right]
\Mod{N\cdot M_{2g}(\mathbb{Z})}\right\}
\end{equation*}
with rational Fourier coefficients. If we set
\begin{equation*}
\omega=\left[\begin{matrix}(1/\sqrt{N})I_g&O_g\\O_g&\sqrt{N}I_g\end{matrix}\right],
\end{equation*}
then we know that $\omega\in\mathrm{Sp}_{2g}(\mathbb{R})$ and
\begin{equation*}
\omega\left[\begin{matrix}A&B\\C&D\end{matrix}\right]\omega^{-1}
=\left[\begin{matrix}A&(1/N)B\\NC&D\end{matrix}\right]
\quad
\textrm{for all}~\left[\begin{matrix}A&B\\C&D\end{matrix}\right]\in
\mathrm{Sp}_{2g}(\mathbb{R}).
\end{equation*}
This implies 
\begin{equation*}
\omega\Gamma^1(N)\omega^{-1}=\Gamma_1(N),
\end{equation*}
and so $\mathcal{F}_{1,N}(\mathbb{Q})$ and $\mathcal{F}^1_N(\mathbb{Q})$ are isomorphic via
\begin{eqnarray*}
\mathcal{F}_{1,N}(\mathbb{Q})&\rightarrow&\mathcal{F}^1_N(\mathbb{Q})\\
h(Z)&\mapsto&(h\circ\omega)(Z)=h((1/N)Z).
\end{eqnarray*}
\end{itemize}
\end{remark}

\section {Special values associated with a Siegel family}

As an application of a Siegel family of level $N$
we shall construct a number associated with each ray class modulo $N$
of a CM-field.
\par
Let
$n$ be a positive integer,
$K$ be a CM-field with $[K:\mathbb{Q}]=2n$ and
$\{\varphi_1,\ldots,\varphi_n\}$ be a set of embeddings of $K$ into $\mathbb{C}$ such that $(K,\{\varphi_i\}_{i=1}^n)$ is a CM-type.
We fix a finite Galois extension $L$ of $\mathbb{Q}$ containing $K$, and set
\begin{eqnarray*}
S&=&\{\sigma\in\mathrm{Gal}(L/\mathbb{Q})~|~\sigma|_K=\varphi_i~\textrm{for some $i\in\{1,2,\ldots, n\}$}\},\\
S^*&=&\{\sigma^{-1}~|~\sigma\in S\},\\
H^*&=&\{\gamma\in\mathrm{Gal}(L/\mathbb{Q})~|~\gamma S^*=S^*\}.
\end{eqnarray*}
Let $K^*$ be the subfield of $L$ corresponding to the subgroup $H^*$
of $\mathrm{Gal}(L/\mathbb{Q})$, and let $\{\psi_1,\ldots,\psi_g\}$ be the set of all embeddings of $K^*$ into $\mathbb{C}$ arising from the elements of $S^*$.
Then we know that $(K^*,\{\psi_j\}_{j=1}^g)$ is a primitive CM-type and
\begin{equation*}
K^*=\mathbb{Q}\left(\sum_{i=1}^n a^{\varphi_i}~|~a\in K\right)
\end{equation*}
(\cite[Proposition 28 in $\S$8.3]{Shimura98}).
We call this CM-type $(K^*,\{\psi_j\}_{j=1}^g)$ the reflex of $(K,\{\varphi_i\}_{i=1}^n)$.
Using this CM-type we define an embedding
\begin{eqnarray*}
\Psi:K^*&\rightarrow&\mathbb{C}^g\\
a&\mapsto&\left[\begin{matrix}a^{\psi_1}\\\vdots\\
a^{\psi_g}\end{matrix}\right].
\end{eqnarray*}
For each purely imaginary element $c$ of $K^*$ we associate an $\mathbb{R}$-bilinear form
\begin{equation*}
\begin{array}{cccl}
E_c:&\mathbb{C}^g\times\mathbb{C}^g&\rightarrow&\mathbb{R}\\
&(\mathbf{u},\mathbf{v})&\mapsto&\displaystyle\sum_{j=1}^g c^{\psi_j}(u_j\overline{v}_j-
\overline{u}_jv_j)\quad(\mathbf{u}=\left[\begin{matrix}u_1\\\vdots\\u_g\end{matrix}\right],
\mathbf{v}=\left[\begin{matrix}v_1\\\vdots\\v_g\end{matrix}\right]).
\end{array}
\end{equation*}
Then, one can readily check that
\begin{equation}\label{Tr}
E_c(\Psi(a),\Psi(b))=\mathrm{Tr}_{K^*/\mathbb{Q}}(ca\overline{b})
\quad\textrm{for all}~a,b\in K^*
\end{equation}
by utilizing the fact $\overline{a^{\psi_j}}=\overline{a}^{\psi_j}$ for all $a\in K^*$
($1\leq j\leq g$).

\begin{assumption}\label{assumption}
In what follows we assume the following conditions:
\begin{itemize}
\item[(i)] $(K^*)^*=K$.
\item[(ii)] There is a purely imaginary element $\xi$ of $K^*$
and a $\mathbb{Z}$-basis $\{\mathbf{a}_1,\ldots,\mathbf{a}_{2g}\}$
of the lattice $\Psi(\mathcal{O}_{K^*})$ in $\mathbb{C}^g$ for which
\begin{equation*}
\left[\begin{matrix}E_\xi(\mathbf{a}_i,\mathbf{a}_j)\end{matrix}\right]_{1\leq i,j\leq 2g}
=\left[\begin{matrix}O_g&-I_g\\I_g&O_g\end{matrix}\right].
\end{equation*}
In this case, we say that the complex torus
$(\mathbb{C}^g/\Psi(\mathcal{O}_{K^*}),E_\xi)$ is a principally polarized abelian variety
with a symplectic basis $\{\mathbf{a}_1,\ldots,\mathbf{a}_{2g}\}$.
See \cite[$\S$6.2]{Shimura98}.
\item[(iii)] $\mathfrak{f}=N\mathcal{O}_K$ for an integer $N\geq2$.
\end{itemize}
\end{assumption}

\begin{remark}
The Assumption \ref{assumption} (i) is equivalent to saying that $(K,\{\varphi_i\}_{i=1}^n)$ is a primitive CM-type, namely, the abelian varieties of this CM-type are simple (\cite[Proposition 26 in $\S$8.2]{Shimura98}).
\end{remark}

By Assumption \ref{assumption} (i) one can define a group homomorphism
\begin{equation*}
\begin{array}{cccc}
\mathfrak{g}:&K^\times&\rightarrow& (K^*)^\times\\
&d&\mapsto&\displaystyle\prod_{i=1}^n d^{\varphi_i},
\end{array}
\end{equation*}
and extend it continuously to the homomorphism
$\mathfrak{g}: K^\times_\mathbb{A}\rightarrow (K^*)_\mathbb{A}^\times$
of idele groups.
It is also known that for each fractional ideal $\mathfrak{a}$ of $K$
there is a fractional ideal $\mathcal{G}(\mathfrak{a})$ of $K^*$ such that
\begin{equation*}
\mathcal{G}(\mathfrak{a})\mathcal{O}_L
=\prod_{i=1}^n(\mathfrak{a}\mathcal{O}_L)^{\varphi_i}
\end{equation*}
(\cite[$\S$8.3]{Shimura98}).
Let $\mathcal{C}$ be a given ray class in $\mathrm{Cl}(\mathfrak{f})$. Take any integral
ideal $\mathfrak{c}$ in $\mathcal{C}$, and let
\begin{equation*}
\mathcal{N}(\mathfrak{c})=\mathcal{N}_{K/\mathbb{Q}}(\mathfrak{c})=|\mathcal{O}_K/\mathfrak{c}|.
\end{equation*}

\begin{lemma}\label{also}
$(\mathbb{C}^g/\Psi(\mathcal{G}(\mathfrak{c})^{-1}),E_{\xi \mathcal{N}(\mathfrak{c})})$
is also a principally polarized abelian variety.
\end{lemma}
\begin{proof}
It follows from (\ref{Tr}) that
\begin{eqnarray*}
E_{\xi \mathcal{N}(\mathfrak{c})}(\Psi(\mathcal{G}(\mathfrak{c})^{-1}),
\Psi(\mathcal{G}(\mathfrak{c})^{-1}))&=&
\mathrm{Tr}_{K^*/\mathbb{Q}}(\xi \mathcal{N}(\mathfrak{c})\mathcal{G}(\mathfrak{c})^{-1}
\overline{\mathcal{G}(\mathfrak{c})^{-1}})\\
&=&\mathrm{Tr}_{K^*/\mathbb{Q}}(\xi\mathcal{O}_{K^*})\\
&=&E_\xi(\Psi(\mathcal{O}_{K^*}),\Psi(\mathcal{O}_{K^*}))\\
&\subseteq&\mathbb{Z}
\end{eqnarray*}
because $E_\xi$ is a Riemann form
on $\mathbb{C}^{g\phantom{\big|}\hspace{-0.1cm}}/\Psi(\mathcal{O}_{K^*})$.
Thus $E_{\xi \mathcal{N}(\mathfrak{c})}$ defines a Riemann form on $\mathbb{C}^{g\phantom{\big|}\hspace{-0.1cm}}/\Psi(\mathcal{G}(\mathfrak{c})^{-1})$.
\par
Now, let $\{\mathbf{b}_1,\ldots,\mathbf{b}_{2g}\}$ be a symplectic basis of the abelian variety $(\mathbb{C}^{g\phantom{\big|}\hspace{-0.1cm}}/\Psi(\mathcal{G}(\mathfrak{c})^{-1}),E_{\xi \mathcal{N}(\mathfrak{c})})$
so that
\begin{equation*}
\Psi(\mathcal{G}(\mathfrak{c})^{-1})=
\sum_{j=1}^{2g}\mathbb{Z}\mathbf{b}_j\quad
\textrm{and}\quad
\left[\begin{matrix}E_{\xi \mathcal{N}(\mathfrak{c})}(\mathbf{b}_i,
\mathbf{b}_j)
\end{matrix}\right]_{1\leq i,j\leq 2g}=\left[\begin{matrix}O_g&-\mathcal{E}\\
\mathcal{E}& O_g\end{matrix}\right],
\end{equation*}
where $\mathcal{E}=\left[\begin{matrix}\varepsilon_1 & \cdots & 0\\
\vdots & \ddots & \vdots\\
0 & \cdots & \varepsilon_g\end{matrix}\right]$
is a $g\times g$ diagonal matrix
for some
positive integers $\varepsilon_1,\ldots,\varepsilon_g$ satisfying $\varepsilon_1\,|\,\cdots\,|\,\varepsilon_g$.
Furthermore, let $b_1\ldots,b_{2g}$ be elements of $\mathcal{G}(\mathfrak{c})^{-1}$ such that $\mathbf{b}_j=\Psi(b_j)$ ($1\leq j\leq 2g$).
Since $\mathcal{O}_{K^*}\subseteq\mathcal{G}(\mathfrak{c})^{-1}$, we have
\begin{equation}\label{abaa}
\left[\begin{matrix}
\mathbf{a}_1 & \cdots & \mathbf{a}_{2g}
\end{matrix}\right]
=\left[\begin{matrix}
\mathbf{b}_1 & \cdots & \mathbf{b}_{2g}
\end{matrix}\right]\alpha\quad\textrm{for some}~\alpha\in M_{2g}(\mathbb{Z})\cap\mathrm{GL}_{2g}(\mathbb{Q}),
\end{equation}
and hence
\begin{equation*}
\left[\begin{matrix}
a_1^{\psi_1} & \cdots & a_{2g}^{\psi_1}\\
\vdots &  & \vdots\\
a_1^{\psi_g} & \cdots & a_{2g}^{\psi_g}\\
\overline{a_1^{\psi_1}} & \cdots & \overline{a_{2g}^{\psi_1}}\\
\vdots &  & \vdots\\
\overline{a_1^{\psi_g}} & \cdots & \overline{a_{2g}^{\psi_g}}
\end{matrix}\right]=
\left[\begin{matrix}
b_1^{\psi_1} & \cdots & b_{2g}^{\psi_1}\\
\vdots & & \vdots\\
b_1^{\psi_g} & \cdots & b_{2g}^{\psi_g}\\
\overline{b_1^{\psi_1}} & \cdots & \overline{b_{2g}^{\psi_1}}\\
\vdots & & \vdots\\
\overline{b_1^{\psi_g}} & \cdots & \overline{b_{2g}^{\psi_g}}
\end{matrix}\right]\alpha.
\end{equation*}
Taking determinant and squaring gives rise to the identity
\begin{equation*}
\Delta_{K^*/\mathbb{Q}}(a_1,\ldots,a_{2g})=\Delta_{K^*/\mathbb{Q}}(b_1,\ldots,b_{2g})\det(\alpha)^2.
\end{equation*}
It then follows that
\begin{equation}\label{detam}
\begin{array}{ccl}\displaystyle
\det(\alpha)^2=
\frac{|\Delta_{K^*/\mathbb{Q}}(a_1,\ldots,a_{2g})|}{|\Delta_{K^*/\mathbb{Q}}(b_1,\ldots,b_{2g})|}=
\frac{d_{K^*/\mathbb{Q}}(\mathcal{O}_{K^*})}
{d_{K^*/\mathbb{Q}}(\mathcal{G}(\mathfrak{c})^{-1})}
&=&\mathcal{N}_{K^*/\mathbb{Q}}(\mathcal{G}(\mathfrak{c}))^2\\
&=&\mathcal{N}_{K^*/\mathbb{Q}}(\mathcal{G}(\mathfrak{c})\overline{\mathcal{G}(\mathfrak{c})})\vspace{0.1cm}\\
&=&\mathcal{N}(\mathfrak{c})^{2g},
\end{array}
\end{equation}
where $d_{K^*/\mathbb{Q}}$ stands for the discriminant of a fractional ideal of $K^*$ (\cite[Proposition 13 in Chapter III]{Lang}).
And, we deduce by (\ref{abaa}) that
\begin{eqnarray*}
\mathcal{N}(\mathfrak{c})\left[\begin{matrix}O_g & -I_g\\I_g&O_g\end{matrix}\right]&=&
\left[\mathcal{N}(\mathfrak{c})\begin{matrix}E_\xi(\mathbf{a}_i,\mathbf{a}_j)\end{matrix}\right]_{1\leq i,j\leq 2g}\\
&=&
\left[\begin{matrix}E_{\xi \mathcal{N}(\mathfrak{c})}(\mathbf{a}_i,\mathbf{a}_j)\end{matrix}\right]_{1\leq i,j\leq 2g}\\
&=&\alpha^T\left[\begin{matrix}E_{\xi \mathcal{N}(\mathfrak{c})}(\mathbf{b}_i,\mathbf{b}_j)\end{matrix}\right]_{1\leq i,j\leq 2g}\alpha\\
&=&\alpha^T\left[\begin{matrix}O_g&-\mathcal{E}\\
\mathcal{E}&O_g\end{matrix}\right]\alpha.
\end{eqnarray*}
By taking determinant we get
\begin{equation*}
\mathcal{N}(\mathfrak{c})^{2g}=\det(\alpha)^2(\varepsilon_1\cdots\varepsilon_g)^2,
\end{equation*}
which yields by (\ref{detam}) that $\varepsilon_1=\cdots=\varepsilon_g=1$, and so $\mathcal{E}=I_g$.
Therefore, $(\mathbb{C}^{g\phantom{\big|}\hspace{-0.1cm}}/\Psi(\mathcal{G}(\mathfrak{c})^{-1}),
E_{\xi \mathcal{N}(\mathfrak{c})})$ becomes a principally polarized abelian variety.
\end{proof}

As in the proof of Lemma \ref{also} we take a symplectic basis $\{\mathbf{b}_1,\ldots,\mathbf{b}_{2g}\}$
of the principally polarized abelian variety
$(\mathbb{C}^g/\Psi(\mathcal{G}(\mathfrak{c})^{-1}),E_{\xi \mathcal{N}(\mathfrak{c})})$, and
let $b_1,\ldots,b_{2g}$ be elements of $\mathcal{G}(\mathfrak{c})^{-1}$
such that $\mathbf{b}_j=\Psi(b_j)$ ($1\leq j\leq 2g$).
We then have
\begin{equation}\label{aba}
\left[\begin{matrix}\mathbf{a}_1 & \cdots & \mathbf{a}_{2g}\end{matrix}\right]=
\left[\begin{matrix}\mathbf{b}_1 & \cdots & \mathbf{b}_{2g}\end{matrix}\right]\alpha
\quad\textrm{for some}~\alpha=\left[\begin{matrix}
A&B\\C&D\end{matrix}\right]\in M_{2g}(\mathbb{Z})\cap\mathrm{GSp}_{2g}(\mathbb{Q}).
\end{equation}
Since $\nu(\alpha)=\mathcal{N}(\mathfrak{c})$ is relatively prime to $N$,
the reduction $\widetilde{\alpha}$ of $\alpha$ modulo $N$ belongs to
$\mathrm{GSp}_{2g}(\mathbb{Z}/N\mathbb{Z})$.
Let $Z_\mathfrak{c}^*$ be the CM-point
associated with the symplectic basis $\{\mathbf{b}_1,\ldots,\mathbf{b}_{2g}\}$, namely
\begin{equation*}
Z_\mathfrak{c}^*=\left[\begin{matrix}\mathbf{b}_{g+1} & \cdots & \mathbf{b}_{2g}\end{matrix}\right]^{-1}
\left[\begin{matrix}\mathbf{b}_1 & \cdots & \mathbf{b}_g\end{matrix}\right]
\end{equation*}
which belongs to $\mathbb{H}_g$ (\cite[Proposition 8.1.1]{B-L}).

\begin{definition}\label{invariant}
Let $\{h_M(Z)\}_M\in\mathcal{S}_N$.
For a given ray class $\mathcal{C}\in\mathrm{Cl}(\mathfrak{f})$ we define
\begin{equation*}
h_\mathfrak{f}(\mathcal{C})=h_{(1/N)\left[\begin{smallmatrix}B\\D\end{smallmatrix}\right]}(Z_\mathfrak{c}^*).
\end{equation*}
\end{definition}

\begin{remark}
Here, the index matrix $(1/N)\left[\begin{matrix}B\\D\end{matrix}\right]$
is obtained by the fact 
\begin{equation*}
(\left[\begin{matrix}O_g&-I_g\\I_g&O_g\end{matrix}\right]\alpha)^T
=\left[\begin{matrix}B^T &D^T\\-A^T&-C^T\end{matrix}\right].
\end{equation*}
\end{remark}

\section {Well-definedness of $h_\mathfrak{f}(\mathcal{C})$}

In this section we shall show that the value $h_\mathfrak{f}(\mathcal{C})$ given in Definition \ref{invariant} depends only on the ray class $\mathcal{C}$, and hence it is independent of
the choice of a symplectic basis and an integral ideal in $\mathcal{C}$.

\begin{proposition}\label{indep1}
$h_\mathfrak{f}(\mathcal{C})$ does not depend on the choice of
a symplectic basis $\{\mathbf{b}_1,\ldots,\mathbf{b}_{2g}\}$ of
$(\mathbb{C}^g/\Psi(\mathcal{G}(\mathfrak{c})^{-1}), E_{\xi \mathcal{N}(\mathfrak{c})})$.
\end{proposition}
\begin{proof}
Let $\{\widehat{\mathbf{b}}_1,\ldots,\widehat{\mathbf{b}}_{2g}\}$ be another
symplectic basis of $(\mathbb{C}^g/\Psi(\mathcal{G}(\mathfrak{c})^{-1}), E_{\xi \mathcal{N}(\mathfrak{c})})$, and so
\begin{equation}\label{beta}
\left[\begin{matrix}\widehat{\mathbf{b}}_1 & \cdots & \widehat{\mathbf{b}}_{2g}\end{matrix}\right]=
\left[\begin{matrix}
\mathbf{b}_1 & \cdots & \mathbf{b}_{2g}
\end{matrix}\right]\beta\quad\textrm{for some}~\beta=\left[\begin{matrix}P&Q\\
R&S\end{matrix}\right]\in\mathrm{GL}_{2g}(\mathbb{Z}).
\end{equation}
We then derive that
\begin{equation*}
\left[\begin{matrix}O_g&-I_g\\I_g&O_g\end{matrix}\right]
=\left[\begin{matrix}
E_{\xi \mathcal{N}(\mathfrak{c})}(\widehat{\mathbf{b}}_i,\widehat{\mathbf{b}}_j)
\end{matrix}\right]_{1\leq i,j\leq 2g}
=\beta^T\left[\begin{matrix}
E_{\xi \mathcal{N}(\mathfrak{c})}(\mathbf{b}_i,\mathbf{b}_j)
\end{matrix}\right]_{1\leq i,j\leq 2g}\beta
=\beta^T\left[\begin{matrix}O_g&-I_g\\I_g&O_g\end{matrix}\right]\beta,
\end{equation*}
which shows that $\beta\in\mathrm{Sp}_{2g}(\mathbb{Z})$.
Since
\begin{equation*}
\left[\begin{matrix}\mathbf{a}_1 & \cdots & \mathbf{a}_{2g}\end{matrix}\right]=
\left[\begin{matrix}\mathbf{b}_1 & \cdots & \mathbf{b}_{2g}\end{matrix}\right]\alpha
=\left[\begin{matrix}\widehat{\mathbf{b}}_1 & \cdots & \widehat{\mathbf{b}}_{2g}\end{matrix}\right]\beta^{-1}\alpha
\end{equation*}
by (\ref{aba}) and (\ref{beta}),
the special value
obtained by $\{\widehat{\mathbf{b}}_1,\ldots,\widehat{\mathbf{b}}_{2g}\}$
is
\begin{equation*}
h_{(1/N)\beta^{-1}\left[\begin{smallmatrix}B\\D\end{smallmatrix}\right]}
(\widehat{Z}_\mathfrak{c}^*),
\end{equation*}
where $\widehat{Z}_\mathfrak{c}^*$ is the CM-point
corresponding to $\{\widehat{\mathbf{b}}_1,\ldots,\widehat{\mathbf{b}}_{2g}\}$.
\par
On the other hand, we attain that
\begin{eqnarray}
\widehat{Z}_\mathfrak{c}^*&=&
\left[\begin{matrix}
\widehat{\mathbf{b}}_{g+1} & \cdots & \widehat{\mathbf{b}}_{2g}\end{matrix}\right]^{-1}
\left[\begin{matrix}
\widehat{\mathbf{b}}_1 & \cdots & \widehat{\mathbf{b}}_g\end{matrix}\right]\nonumber\\
&=&\left(\left[\begin{matrix}\mathbf{b}_1 & \cdots &\mathbf{b}_g\end{matrix}\right]Q+
\left[\begin{matrix}\mathbf{b}_{g+1} & \cdots &\mathbf{b}_{2g}\end{matrix}\right]S\right)^{-1}\nonumber\\
&&
\left(\left[\begin{matrix}\mathbf{b}_1 & \cdots &\mathbf{b}_g\end{matrix}\right]P+
\left[\begin{matrix}\mathbf{b}_{g+1} & \cdots &\mathbf{b}_{2g}\end{matrix}\right]R\right)\quad\textrm{by (\ref{beta})}\nonumber\\
&=&
\left(P^T\left[\begin{matrix}\mathbf{b}_1 & \cdots &\mathbf{b}_g\end{matrix}\right]^T+
R^T\left[\begin{matrix}\mathbf{b}_{g+1} & \cdots &\mathbf{b}_{2g}\end{matrix}\right]^T\right)\nonumber\\
&&\left(Q^T\left[\begin{matrix}\mathbf{b}_1 & \cdots &\mathbf{b}_g\end{matrix}\right]^T+
S^T\left[\begin{matrix}\mathbf{b}_{g+1} & \cdots &\mathbf{b}_{2g}\end{matrix}\right]^T\right)^{-1}\quad\textrm{since}~(\widehat{Z}_\mathfrak{c}^*)^T=
\widehat{Z}_\mathfrak{c}^*\nonumber\\
&=&\left(P^T\left(\left[\begin{matrix}\mathbf{b}_{g+1} & \cdots &\mathbf{b}_{2g}\end{matrix}\right]^{-1}
\left[\begin{matrix}\mathbf{b}_1 & \cdots &\mathbf{b}_g\end{matrix}\right]\right)^T+
R^T\right)\nonumber\\
&&\left(Q^T\left(\left[\begin{matrix}\mathbf{b}_{g+1} & \cdots &\mathbf{b}_{2g}\end{matrix}\right]^{-1}
\left[\begin{matrix}\mathbf{b}_1 & \cdots &\mathbf{b}_g\end{matrix}\right]\right)^T+S^T\right)^{-1}\nonumber\\
&=&(P^T(Z_\mathfrak{c}^*)^T+R^T)(Q^T(Z_\mathfrak{c}^*)^T+S^T)^{-1}\nonumber\\
&=&(P^TZ_\mathfrak{c}^*+R^T)(Q^TZ_\mathfrak{c}^*+S^T)^{-1}\quad\textrm{because}~(Z_\mathfrak{c}^*)^T=Z_\mathfrak{c}^*\nonumber\\
&=&\beta^T(Z_\mathfrak{c}^*).\label{Ztilde}
\end{eqnarray}
Thus we deduce that
\begin{eqnarray*}
h_{(1/N)\beta^{-1}\left[\begin{smallmatrix}B\\D\end{smallmatrix}\right]}(\widehat{Z}_\mathfrak{c}^*)
&=&h_{(1/N)\beta^{-1}\left[\begin{smallmatrix}B\\D\end{smallmatrix}\right]}(\beta^T(Z_\mathfrak{c}^*))
\quad\textrm{by (\ref{Ztilde})}\\
&=&(h_{(1/N)\beta^{-1}\left[\begin{smallmatrix}B\\D\end{smallmatrix}\right]}(Z))^{\beta^T}|_{Z=Z_\mathfrak{c}^*}\\
&=&h_{(1/N)(\beta^T)^T\beta^{-1}\left[\begin{smallmatrix}B\\D\end{smallmatrix}\right]}(Z_\mathfrak{c}^*)
\quad\textrm{by the property (S3) of $\{h_M(Z)\}_M$}\\
&=&h_{(1/N)\left[\begin{smallmatrix}B\\D\end{smallmatrix}\right]}(Z_\mathfrak{c}^*).
\end{eqnarray*}
This proves that $h_\mathfrak{f}(\mathcal{C})$ is independent of the choice of a symplectic basis
of $(\mathbb{C}^g/\Psi(\mathcal{G}(\mathfrak{c})^{-1}), E_{\xi \mathcal{N}(\mathfrak{c})})$.
\end{proof}

\begin{remark}
In like manner one can readily show that
$h_\mathfrak{f}(\mathcal{C})$ does not depend on the choice of
a symplectic basis $\{\mathbf{a}_1,\ldots,\mathbf{a}_{2g}\}$
of $(\mathbb{C}^g/\Psi(\mathcal{O}_K),E_\xi)$.
\end{remark}

\begin{proposition}\label{choiceofbasis}
$h_\mathfrak{f}(\mathcal{C})$ does not depend on the choice of an integral ideal $\mathfrak{c}$ in $\mathcal{C}$.
\end{proposition}
\begin{proof}
Let $\mathfrak{c}'$ be another integral ideal in the class $\mathcal{C}$, and hence
\begin{equation}\label{c'c}
\mathfrak{c}'\mathfrak{c}^{-1}=(1+a)\mathcal{O}_K\quad\textrm{for some}~a\in
\mathfrak{f}\mathfrak{a}^{-1},
\end{equation}
where $\mathfrak{a}$ is an integral ideal of $K$ relatively prime to $\mathfrak{f}$.
Since $1\in\mathfrak{c}^{-1}$
and $(1+a)\in\mathfrak{c}'\mathfrak{c}^{-1}\subseteq\mathfrak{c}^{-1}$, we get
$a\in\mathfrak{c}^{-1}$.
Thus we derive that
\begin{eqnarray*}
a\mathfrak{a}\mathfrak{c}&\subseteq&\mathfrak{f}\mathfrak{c}\cap\mathfrak{a}
\quad\textrm{by the facts $a\in\mathfrak{f}\mathfrak{a}^{-1}$ and $a\in\mathfrak{c}^{-1}$}\\
&\subseteq&\mathfrak{f}\cap\mathfrak{a}\\
&=&\mathfrak{f}\mathfrak{a}\quad\textrm{because $\mathfrak{f}$ and $\mathfrak{a}$ are relatively prime},
\end{eqnarray*}
from which it follows that $a\in\mathfrak{f}\mathfrak{c}^{-1}$.
We then achieve by the fact $\mathfrak{f}=N\mathcal{O}_K$ that
\begin{equation}\label{g(1+a)}
\mathfrak{g}(1+a)
=\prod_{i=1}^n (1+a)^{\varphi_i}\in K^*\cap\prod_{i=1}^n (1+N(\mathfrak{c}^{-1}\mathcal{O}_L)^{\varphi_i})
\subseteq
K^*\cap(1+N\mathcal{G}(\mathfrak{c})^{-1}\mathcal{O}_L)
=1+N\mathcal{G}(\mathfrak{c})^{-1}.
\end{equation}
\par
Let
\begin{equation}\label{b'b'}
b_j'=\mathfrak{g}(1+a)^{-1}b_j\quad\textrm{and}\quad
\mathbf{b}_j'=\Psi(b_j')\quad(1\leq j\leq 2g).
\end{equation}
We know that $\{\mathbf{b}_1',\ldots,\mathbf{b}_{2g}'\}$
is a $\mathbb{Z}$-basis of the lattice $\Psi(\mathcal{G}(\mathfrak{c}')^{-1})$ in $\mathbb{C}^g$
and
\begin{equation}\label{bDb}
\mathbf{b}_j'=T\mathbf{b}_j\quad\textrm{with}~
T=\left[
\begin{matrix}
(\mathfrak{g}(1+a)^{-1})^{\psi_1} & \cdots & 0\\
\vdots & \ddots & \vdots\\
0 & \cdots & (\mathfrak{g}(1+a)^{-1})^{\psi_g}
\end{matrix}
\right].
\end{equation}
Furthermore, we get that
\begin{eqnarray*}
\left[\begin{matrix}
E_{\xi \mathcal{N}(\mathfrak{c}')}(\mathbf{b}_i',\mathbf{b}_j')
\end{matrix}\right]_{1\leq i,j\leq 2g}
&=&\left[\begin{matrix}
\mathrm{Tr}_{K^*/\mathbb{Q}}(\xi \mathcal{N}(\mathfrak{c}')b_i'\overline{b_j'})
\end{matrix}\right]_{1\leq i,j\leq 2g}\quad\textrm{by (\ref{Tr})}\\
&=&\left[\begin{matrix}
\mathrm{Tr}_{K^*/\mathbb{Q}}(\xi \mathcal{N}(\mathfrak{c}')
\mathfrak{g}(1+a)^{-1}b_i\overline{\mathfrak{g}(1+a)^{-1}b_j})
\end{matrix}\right]_{1\leq i,j\leq 2g}\quad\textrm{by (\ref{b'b'})}\\
&=&\left[\begin{matrix}
\mathrm{Tr}_{K^*/\mathbb{Q}}(\xi \mathcal{N}(\mathfrak{c}')\mathrm{N}_{K/\mathbb{Q}}(1+a)^{-1}
b_i\overline{b_j})
\end{matrix}\right]_{1\leq i,j\leq 2g}\\
&=&\left[\begin{matrix}
\mathrm{Tr}_{K/\mathbb{Q}}(\xi \mathcal{N}(\mathfrak{c}) b_i\overline{b_j})
\end{matrix}\right]_{1\leq i,j\leq 2g}\\
&&\textrm{by (\ref{c'c}) and the fact $\mathrm{N}_{K/\mathbb{Q}}(1+a)>0$}\\
&=&\left[\begin{matrix}E_{\xi \mathcal{N}(\mathfrak{c})}(\mathbf{b}_i,\mathbf{b}_j)
\end{matrix}\right]_{1\leq i,j\leq 2g}\quad\textrm{by (\ref{Tr})}\\
&=&\left[\begin{matrix}
O_g&-I_g\\
I_g&O_g\end{matrix}\right].
\end{eqnarray*}
Thus $\{\mathbf{b}_1',\ldots,\mathbf{b}_{2g}'\}$ is a symplectic basis of $(\mathbb{C}^{g\phantom{\big|}\hspace{-0.1cm}}/\Psi(\mathcal{G}(\mathfrak{c'})^{-1}),E_{\xi \mathcal{N}(\mathfrak{c}')})$, and its associated CM-point $Z_{\mathfrak{c}'}^*$ is given by
\begin{eqnarray}
Z_{\mathfrak{c}'}^*&=&\left[\begin{matrix}\mathbf{b}_{g+1}' & \cdots
&\mathbf{b}_{2g}'\end{matrix}\right]^{-1}
\left[\begin{matrix}\mathbf{b}_1' & \cdots
&\mathbf{b}_g'\end{matrix}\right]\nonumber\\
&=&\left[\begin{matrix}T\mathbf{b}_{g+1} & \cdots
&T\mathbf{b}_{2g}\end{matrix}\right]^{-1}
\left[\begin{matrix}T\mathbf{b}_1 & \cdots
&T\mathbf{b}_g\end{matrix}\right]\quad\textrm{by (\ref{bDb})}\nonumber\\
&=&Z_\mathfrak{c}^*.\label{Zc'}
\end{eqnarray}
\par
Let $\alpha=\left[\begin{matrix}a_{ij}\end{matrix}\right]$,
$\alpha'=\left[\begin{matrix}a_{ij}'\end{matrix}\right]\in M_{2g}(\mathbb{Z})$ such that
\begin{equation}\label{Nababa}
\left[\begin{matrix}\mathbf{a}_1 & \cdots & \mathbf{a}_{2g}\end{matrix}\right]
=\left[\begin{matrix}\mathbf{b}_1 & \cdots & \mathbf{b}_{2g}\end{matrix}\right]\alpha=
\left[\begin{matrix}\mathbf{b}_1' & \cdots & \mathbf{b}_{2g}'\end{matrix}\right]\alpha'.
\end{equation}
For each $1\leq i\leq 2g$ we obtain that
\begin{eqnarray*}
\sum_{j=1}^{2g}a_{ji}'b_j&=&\mathfrak{g}(1+a)\sum_{j=1}^{2g}a_{ji}'b_j'\quad\textrm{by (\ref{b'b'})}\\
&=&a_i\mathfrak{g}(1+a)\quad\textrm{by (\ref{Nababa})}\\
&\in&a_i(1+N\mathcal{G}(\mathfrak{c})^{-1})\quad\textrm{by (\ref{g(1+a)})}\\
&\subseteq&a_i+N\mathcal{G}(\mathfrak{c})^{-1}\quad\textrm{because}~ a_i\in\mathcal{O}_K\\
&=&\sum_{j=1}^{2g}a_{ji}b_j+N\sum_{j=1}^{2g}\mathbb{Z}b_j
\quad\textrm{by (\ref{Nababa})}.
\end{eqnarray*}
This yields $\alpha\equiv \alpha'\Mod{N\cdot M_{2g}(\mathbb{Z})}$, and hence
\begin{equation}\label{NaNa'}
(1/N)\alpha\equiv(1/N)\alpha'\Mod{M_{2g}(\mathbb{Z})}.
\end{equation}
Now, the result follows from (\ref{Zc'}), (\ref{NaNa'}) and the property (S2)
of $\{h_M(Z)\}_M$.
\end{proof}

\section {Galois actions on $h_\mathfrak{f}(\mathcal{C})$}

Finally we shall show that if $h_\mathfrak{f}(\mathcal{C})$ is finite, then it lies in the ray class field $K_\mathfrak{f}$ and satisfies the natural transformation formula under the Artin reciprocity map for $\mathfrak{f}$.
\par

Let $r:K^*\rightarrow M_{2g}(\mathbb{Q})$
be the regular representation with respect to the ordered basis
$\{a_1,\ldots,a_{2g}\}$ of $K^*$ over $\mathbb{Q}$ given by
\begin{equation}\label{hayay}
a\left[\begin{matrix}
a_1 \\ \vdots \\ a_{2g}
\end{matrix}\right]
=r(a)\left[\begin{matrix}
a_1 \\ \vdots \\ a_{2g}
\end{matrix}\right]\quad(a\in K^*).
\end{equation}
Then it can be extended to the map $r:(K^*)_\mathbb{A}\rightarrow M_{2g}(\mathbb{Q}_\mathbb{A})$ of adele rings.

\begin{lemma}[Shimura's Reciprocity Law]\label{reciprocity}
Let $f$ be an element of $\mathcal{F}$ which is finite at $Z_\mathfrak{c}^*$.
\begin{itemize}
\item[\textup{(i)}] The special value $f(Z_\mathfrak{c}^*)$ lies in $K_\mathrm{ab}$.
\item[\textup{(ii)}] For every $s\in K_\mathbb{A}^\times$ we have
$r(\mathfrak{g}(s))\in G_{\mathbb{A}+}$ and
\begin{equation*}
f(Z_\mathfrak{c}^*)^{[s,K]}=f^{\tau(r(\mathfrak{g}(s)^{-1}))}(Z_\mathfrak{c}^*).
\end{equation*}
\end{itemize}
\end{lemma}
\begin{proof}
See \cite[Lemma 9.5 and Theorem 9.6]{Shimura00}.
\end{proof}

\begin{theorem}\label{main}
If $h_\mathfrak{f}(\mathcal{C})$ is finite, then it belongs to $K_\mathfrak{f}$. 
And it satisfies \begin{equation*}
h_\mathfrak{f}(\mathcal{C})^{\sigma_\mathfrak{f}(\mathcal{D})}=
h_\mathfrak{f}(\mathcal{CD})\quad\textrm{for every}~\mathcal{D}\in\mathrm{Cl}(\mathfrak{f}),
\end{equation*}
where $\sigma_\mathfrak{f}$ is the Artin reciprocity map for $\mathfrak{f}$.
\end{theorem}
\begin{proof}
Since $h_\mathfrak{f}(\mathcal{C})$ belongs to $K_\mathrm{ab}$ by Lemma \ref{reciprocity} (i), there is a sufficiently large positive integer $M$ so that $N\,|\,M$ and $h_\mathfrak{f}(\mathcal{C})\in K_\mathfrak{m}$ with $\mathfrak{m}=M\mathcal{O}_K$.
Take an integral ideal $\mathfrak{d}$ in $\mathcal{D}$
relatively prime to $\mathfrak{m}$ by using the surjectivity of the natural map
$\mathrm{Cl}(\mathfrak{m})\rightarrow\mathrm{Cl}(\mathfrak{f})$.
Let $\{\mathbf{d}_1,\ldots,\mathbf{d}_{2g}\}$
be a symplectic basis of the principally polarized abelian variety
$(\mathbb{C}^{g\phantom{\big|}\hspace{-0.1cm}}/\Psi(
\mathcal{G}(\mathfrak{c}\mathfrak{d})^{-1}),
E_{\xi \mathcal{N}(\mathfrak{cd})})$, and let
$d_1,\ldots,d_{2g}$
be elements of $\mathcal{G}(\mathfrak{c}\mathfrak{d})^{-1}$
such that $\mathbf{d}_j=\Psi(d_j)$
($1\leq j\leq 2g$).
Since $\mathcal{G}(\mathfrak{c})^{-1}\subseteq
\mathcal{G}(\mathfrak{c}\mathfrak{d})^{-1}$, we get
\begin{equation}\label{bdd}
\left[\begin{matrix}\mathbf{b}_1 & \cdots & \mathbf{b}_{2g}\end{matrix}\right]=
\left[\begin{matrix}\mathbf{d}_1 & \cdots & \mathbf{d}_{2g}\end{matrix}\right]\delta\quad\textrm{for some}~
\delta\in M_{2g}(\mathbb{Z})\cap\mathrm{GL}_{2g}(\mathbb{Q}).
\end{equation}
We then have that
\begin{eqnarray*}
\left[\begin{matrix}
O_g&-I_g\\I_g&O_g
\end{matrix}\right]&=&
\left[\begin{matrix}
E_{\xi \mathcal{N}(\mathfrak{c})}(\mathbf{b}_i,\mathbf{b}_j)
\end{matrix}\right]_{1\leq i,j\leq 2g}\\
&=&\delta^T\left[\begin{matrix}E_{\xi \mathcal{N}(\mathfrak{c})}(\mathbf{d}_i,\mathbf{d}_j)
\end{matrix}\right]_{1\leq i,j\leq 2g}\delta\quad\textrm{by (\ref{bdd})}\\
&=&\delta^T\left[\begin{matrix}\mathcal{N}(\mathfrak{c})\mathcal{N}(\mathfrak{cd})^{-1}E_{\xi \mathcal{N}(\mathfrak{cd})}(\mathbf{d}_i,\mathbf{d}_j)
\end{matrix}\right]_{1\leq i,j\leq 2g}\delta\\
&=&\mathcal{N}(\mathfrak{d})^{-1}\delta^T
\left[\begin{matrix}
O_g&-I_g\\I_g&O_g
\end{matrix}\right]\delta.
\end{eqnarray*}
This claims that
\begin{equation}\label{deltain}
\delta\in M_{2g}(\mathbb{Z})\cap G_+~
\textrm{with}~\nu(\delta)=
\mathcal{N}(\mathfrak{d}).
\end{equation}
Furthermore, if we let $Z_\mathfrak{cd}^*$ be the CM-point associated with
$\{\mathbf{d}_1,\ldots,\mathbf{d}_{2g}\}$, then
we obtain
\begin{equation}\label{ZdZ}
Z_\mathfrak{cd}^*=(\delta^{-1})^T(Z_\mathfrak{c}^*)
\end{equation}
in a similar way to the argument in the proof of Proposition \ref{indep1}.
\par
Let $s=(s_p)_p$ be an idele of $K$ such that
\begin{equation}\label{idele}
\left\{\begin{array}{rccl}s_p&=&
1 & \textrm{if}~p\,|\,M,\\
s_p(\mathcal{O}_K)_p&=&\mathfrak{d}_p & \textrm{if}~p\nmid M.
\end{array}\right.
\end{equation}
If we set $\widetilde{\mathcal{D}}$ to be the ray class in $\mathrm{Cl}(\mathfrak{m})$ containing $\mathfrak{d}$, then we attain by (\ref{idele}) 
\begin{eqnarray}
[s,K]|_{K_\mathfrak{m}}&=&\sigma_\mathfrak{m}(\widetilde{\mathcal{D}}),\label{sKsD}\\
\mathfrak{g}(s)_p^{-1}(\mathcal{O}_{K^*})_p
&=&\mathcal{G}(\mathfrak{d})^{-1}_p\quad\textrm{for all rational primes $p$}.
\label{gsOscO}
\end{eqnarray}
It then follows from (\ref{hayay})$\sim$(\ref{gsOscO}) that for every rational prime $p$, the entries of each of the vectors
\begin{equation*}
r(\mathfrak{g}(s)^{-1})_p\left[\begin{matrix}b_1\\\vdots\\b_{2g}\end{matrix}\right]
\quad\textrm{and}\quad
(\delta^{-1})^T\left[\begin{matrix}b_1\\\vdots\\b_{2g}\end{matrix}\right]
\end{equation*}
form a basis of $\mathcal{G}(\mathfrak{c}\mathfrak{d})^{-1}_p=
\mathcal{G}(\mathfrak{c})^{-1}\mathcal{G}(\mathfrak{d})^{-1}_p$. So, there is a matrix $u=(u_p)_p\in\prod_p\mathrm{GL}_{2g}(\mathbb{Z}_p)$ satisfying
\begin{equation}\label{hgsud}
r(\mathfrak{g}(s)^{-1})=u(\delta^{-1})^T.
\end{equation}
Since $\delta^T$ and $\left[\begin{matrix}
I_g & O_g\\
O_g & \mathcal{N}(\delta)I_g
\end{matrix}\right]$
can be viewed as elements of $\mathrm{GSp}_{2g}(Z/M\mathbb{Z})$ by (\ref{deltain}),
there exists a matrix $\gamma\in\mathrm{Sp}_{2g}(\mathbb{Z})$ such that
\begin{equation}\label{dIrM}
\delta^T\equiv\left[\begin{matrix}
I_g & O_g\\
O_g & \mathcal{N}(\delta)I_g
\end{matrix}\right]\gamma\Mod{M\cdot M_{2g}(\mathbb{Z})}
\end{equation}
owing to the surjectivity of the reduction $\mathrm{Sp}_{2g}(\mathbb{Z})
\rightarrow\mathrm{Sp}_{2g}(\mathbb{Z}/M\mathbb{Z})$.
Since $r(\mathfrak{g}(s)^{-1})_p=I_{2g}$ for all $p\,|\,M$ by (\ref{idele}), we get
$u_p=\delta^T$ for all $p\,|\,M$ by (\ref{hgsud}). Hence we deduce by (\ref{dIrM}) that
\begin{equation}\label{urIMM}
u_p\gamma^{-1}\equiv\left[\begin{matrix}I_g&O_g\\
O_g&\mathcal{N}(\delta)I_g\end{matrix}\right]\Mod{M\cdot M_{2g}(\mathbb{Z}_p)}\quad
\textrm{for all rational primes $p$}.
\end{equation}
\par
On the other hand, we have by (\ref{aba}) and (\ref{bdd}) that
\begin{equation}\label{newmatrix}
\left[\begin{matrix}\mathbf{a}_1 & \cdots & \mathbf{a}_{2g}\end{matrix}\right]=
\left[\begin{matrix}\mathbf{b}_1 & \cdots & \mathbf{b}_{2g}\end{matrix}\right]\alpha=
(\left[\begin{matrix}\mathbf{b}_1 & \cdots & \mathbf{b}_{2g}\end{matrix}\right]\delta^{-1})(\delta\alpha)=
\left[\begin{matrix}\mathbf{d}_1 & \cdots & \mathbf{d}_{2g}\end{matrix}\right]
(\delta\alpha).
\end{equation}
Letting $\alpha=\left[\begin{matrix}A&B\\C&D\end{matrix}\right]$ we induce that
\begin{eqnarray*}
h_\mathfrak{f}(\mathcal{C})^{\sigma_\mathfrak{m}(\widetilde{\mathcal{D}})}
&=&h_\mathfrak{f}(\mathcal{C})^{[s,K]}\quad\textrm{by (\ref{sKsD})}\\
&=&h_{(1/N)\left[\begin{smallmatrix}B\\D\end{smallmatrix}\right]}(Z_\mathfrak{c}^*)^{[s,K]}\quad\textrm{by Definition \ref{invariant}}\\
&=&h_{(1/N)\left[\begin{smallmatrix}B\\D\end{smallmatrix}\right]}
(Z)^{\tau(r(\mathfrak{g}(s)^{-1}))}|_{Z=Z_\mathfrak{c}^*}\quad\textrm{by Lemma \ref{reciprocity} (ii)}\\
&=&h_{(1/N)\left[\begin{smallmatrix}B\\D\end{smallmatrix}\right]}
(Z)^{\tau(u(\delta^{-1})^T)}|_{Z=Z_\mathfrak{c}^*}\quad\textrm{by (\ref{hgsud})}\\
&=&h_{(1/N)\left[\begin{smallmatrix}B\\D\end{smallmatrix}\right]}
(Z)^{\tau(u\gamma^{-1})\tau(\gamma)\tau((\delta^{-1})^T)}|_{Z=Z_\mathfrak{c}^*}\\
&=&h_{(1/N)\left[\begin{smallmatrix}
I_g&O_g\\
O_g&\mathcal{N}(\delta)I_g\end{smallmatrix}\right]\left[\begin{smallmatrix}B\\D\end{smallmatrix}\right]}(Z)
^{\tau(\gamma)\tau((\delta^{-1})^T)}|_{Z=Z_\mathfrak{c}^*}
\quad\textrm{by (\ref{urIMM}) and (S3)}\\
&=&h_{(1/N)\gamma^T\left[\begin{smallmatrix}
I_g&O_g\\
O_g&\mathcal{N}(\delta)I_g\end{smallmatrix}\right]\left[\begin{smallmatrix}B\\D\end{smallmatrix}\right]}(Z)
^{\tau((\delta^{-1})^T)}|_{Z=Z_\mathfrak{c}^*}
\quad\textrm{by (S3)}\\
&=&h_{(1/N)\delta\left[\begin{smallmatrix}B\\D\end{smallmatrix}\right]}(Z)
^{\tau((\delta^{-1})^T)}|_{Z=Z_\mathfrak{c}^*}
\quad\textrm{by (\ref{dIrM}) and (S2)}\\
&=&h_{(1/N)\delta\left[\begin{smallmatrix}B\\D\end{smallmatrix}\right]}
((\delta^{-1})^T(Z_\mathfrak{c}^*))\quad\textrm{due to the fact
$\delta\in G_+$ and (A1)}\\
&=&h_\mathfrak{f}(\mathcal{CD})\quad\textrm{by (\ref{ZdZ}), (\ref{newmatrix}) and Definition \ref{invariant}}.
\end{eqnarray*}
\par
In particular, suppose that $\mathfrak{d}=d\mathcal{O}_K$ for some
$d\in\mathcal{O}_K$ such that $d\equiv1\Mod{\mathfrak{f}}$. Then
$\mathcal{D}$ is the identity class of $\mathrm{Cl}(\mathfrak{f})$, and so the
above observation implies
that $\sigma_\mathfrak{m}(\widetilde{\mathcal{D}})$ leaves $h_\mathfrak{f}(\mathcal{C})$ fixed.
Therefore, we conclude that $h_\mathfrak{f}(\mathcal{C})$ lies in $K_\mathfrak{f}$.
\end{proof}

\begin{corollary}\label{Cor}
Let $H$ be a subgroup of
$\mathrm{Cl}(\mathfrak{f})$ defined by
\begin{eqnarray*}
H&=&\langle
\mathcal{D}\in\mathrm{Cl}(\mathfrak{f})~|~\textrm{$\mathcal{D}$ contains an integral ideal
$\mathfrak{d}$ of $K$ for which}\\ &&\hspace{2.2cm}\textrm{$\mathcal{G}(\mathfrak{d})=\mathfrak{g}(d)\mathcal{O}_{K^*}$
for some $d\in\mathcal{O}_{K}$ such that $\mathfrak{g}(d)\equiv1\Mod{N\mathcal{O}_{K^*}}$}
\rangle,
\end{eqnarray*}
and let $K_\mathfrak{f}^H$ be the fixed field of $H$.
If $h_\mathfrak{f}(\mathcal{C})$ is finite, then it belongs to $K_\mathfrak{f}^H$.
\end{corollary}
\begin{proof}
Let $\mathcal{C}_0$ be the identity class of $\mathrm{Cl}(\mathfrak{f})$.
Since $h_\mathfrak{f}(\mathcal{C}_0)\in K_\mathfrak{f}$ by Theorem \ref{main}, $K(h_\mathfrak{f}(\mathcal{C}_0))$ is a Galois extension of $K$
as a subfield of $K_\mathfrak{f}$.
Furthermore, since
\begin{equation*}
h_\mathfrak{f}(\mathcal{C}_0)^{\sigma_\mathfrak{f}(\mathcal{C})}=
h_\mathfrak{f}(\mathcal{C}_0\mathcal{C})=
h_\mathfrak{f}(\mathcal{C})
\end{equation*}
by Theorem \ref{main},
$K(h_\mathfrak{f}(\mathcal{C}_0))$ contains $h_\mathfrak{f}(\mathcal{C})$.
Thus it suffices to show that $h_\mathfrak{f}(\mathcal{C}_0)$ belongs to $K_\mathfrak{f}^H$.
\par
To this end, let $\mathcal{D}$ be an element of $\mathrm{Cl}(\mathfrak{f})$
containing an integral ideal $\mathfrak{d}$ of $K$ for which
\begin{equation*}
\mathcal{G}(\mathfrak{d})=\mathfrak{g}(d)\mathcal{O}_{K^*}\quad
\textrm{for some}~d\in\mathcal{O}_K~\textrm{such that}~
\mathfrak{g}(d)\equiv1\Mod{N\mathcal{O}_{K^*}}.
\end{equation*}
Now that
\begin{equation*}
(\mathbb{C}^g/\Psi(\mathcal{G}(\mathfrak{d})^{-1}),
E_{\xi\mathcal{N}(\mathfrak{d})})=
(\mathbb{C}^g/\Psi(\mathfrak{g}(d)^{-1}\mathcal{O}_{K^*}),
E_{\xi\mathcal{N}(d\mathcal{O}_K)}),
\end{equation*}
we obtain
\begin{equation*}
h_\mathfrak{f}(\mathcal{C}_0)^{\sigma_\mathfrak{f}(\mathcal{D})}
=h_\mathfrak{f}(\mathcal{D})=h_\mathfrak{f}([d\mathcal{O}_K]),
\end{equation*}
where $[\mathfrak{a}]$ is the ray class containing $\mathfrak{a}$ for
a fractional ideal $\mathfrak{a}$ of $K$.
Moreover, since $\mathfrak{g}(d)\equiv1\Mod{N\mathcal{O}_{K^*}}$, we achieve 
\begin{equation*}
h_\mathfrak{f}([d\mathcal{O}_K])=
h_\mathfrak{f}([\mathcal{O}_K])=
h_\mathfrak{f}(\mathcal{C}_0)
\end{equation*}
in like manner as in the proof of Proposition \ref{choiceofbasis}. This proves that
$h_\mathfrak{f}(\mathcal{C}_0)$ belongs to $K_\mathfrak{f}^H$.
\end{proof}

\bibliographystyle{amsplain}

\address{% the first author
Department of Mathematical Sciences \\
KAIST \\
Daejeon 34141\\
Republic of Korea} {jkkoo@math.kaist.ac.kr}
\address{% the corresponding author
Department of Mathematics\\
Hankuk University of Foreign Studies\\
Yongin-si, Gyeonggi-do 17035\\
Republic of Korea} {dhshin@hufs.ac.kr}
\address{
Department of Mathematical Sciences \\
KAIST \\
Daejeon 34141\\
Republic of Korea} {math\_dsyoon@kaist.ac.kr}

\end{document}